\documentclass{article}
\usepackage{fullpage,amssymb,amsmath,latexsym}

\def\least{{}_\downarrow}
\def\til#1{{\widetilde #1}}
\def\tilQ{{\widetilde Q}}
\def\val{{\rm val\,}}
\def\Exp{{\rm Exp}}
\def\lam{\lambda}
\def\alp{\alpha}
\def\eqbd{{{:}{=}}}
\def\eps{\epsilon}
\def\bks{\backslash}

\def\bfF{{\bf F}}
\def\setminus{\bks}
\newtheorem{theorem}{Theorem}[section]
\newtheorem{proposition}[theorem]{Proposition}
\newtheorem{lemma}[theorem]{Lemma}

\newtheorem{definition}[theorem]{Definition}
\newtheorem{example}[theorem]{Example}

\newenvironment{@abssec}[1]{%
     \if@twocolumn
       \section*{#1}%
     \else
       \vspace{.05in}\footnotesize
       \parindent .2in
         {\bfseries #1. }\ignorespaces
     \fi}
     {\if@twocolumn\else\par\vspace{.1in}\fi}

\newenvironment{@normsec}[1]{%
     \if@twocolumn
       \section*{#1}%
     \else
       \vspace{.05in}
       \parindent .2in
         {\bfseries #1. }
     \fi}
     {\if@twocolumn\else\par\vspace{.1in}\fi}

\newenvironment{keywords}{\begin{@abssec}{Key words}}{\end{@abssec}}
\newenvironment{AMS}{\begin{@abssec}{AMS subject classification}}{\end{@abssec}}
\newenvironment{proof}{\begin{@normsec}{Proof}}{\hfill $\Box$ \end{@normsec}\medskip }

\newtheorem{remark}[theorem]{Remark}
\newcommand{\R}{\mathbb{R}}
\def\Rn{\R^n}
\def\Zn{\Z^n}
\newcommand{\B}{\mathbb{B}}

\newcommand{\C}{\mathbb{C}}
\newcommand{\Z}{\mathbb{Z}}
\newcommand{\N}{\mathbb{N}}

\newcommand{\Q}{\mathbb{Q}}

\newcommand{\CV}{\mathcal{V}}
\newcommand{\CP}{\mathcal{P}}
\newcommand{\cp}[1]{\mathcal{P}(#1)}
\newcommand{\cpm}[1]{\mathcal{P}_-(#1)}
\newcommand{\cbm}[1]{\B_-(#1)}
\newcommand{\cjm}[1]{\mathcal{J}_-(#1)}
\newcommand{\cim}[1]{\mathcal{I}_-(#1)}
\newcommand{\cdm}[1]{\mathcal{D}_-(#1)}
\newcommand{\CI}{\mathcal{I}}
\newcommand{\CF}{\mathcal{F}}
\newcommand{\CH}{\mathcal{H}}
\newcommand{\CJ}{\mathcal{J}}
\newcommand{\CD}{\mathcal{D}}
\newcommand{\CM}{\mathcal{M}}
\newcommand{\Cb}{\mathcal{V}}
\newcommand{\CZ}{\mathcal{Z}}
\newcommand{\BC}{{\rm Comp}}
\def\Ideal{\mathop{\rm Ideal}\nolimits}
\def\rank{\mathop{\rm rank}\nolimits}
\def\I{\mathbb{I}}
\def\ex{\mathop{\rm ex}}
\def\Sym{\mathop{\rm Sym}}
\def\spam{{\rm span\,}}
\def\codim{{\rm codim\,}}
\def\inpro#1{\langle#1\rangle}
\begin{document}

%%%%%%%%%%%%%%%%%%%%%%%%%%%%%    TITLE   PAGE    %%%%%%%%%%%%%%%%%%%%%%%%%%%%%%
\title{Hierarchical zonotopal spaces}

\author{Olga Holtz\thanks{School of Mathematics, Institute for Advanced Study;
 Departments of Mathematics, University of California-Berkeley and Technische Universit\"at Berlin. The work of
this author was supported by the Sofja Kovalevskaja Research Prize of Alexander von
Humboldt Foundation and by the National Science Foundation under agreement DMS-0635607.
Email: {\tt holtz@math.ias.edu}}
\and
Amos Ron \thanks{Departments of Mathematics and Computer Sciences, University of
 Wisconsin-Madison. The work of this author was supported by the National
 Science Foundation under grants DMS-0602837 and DMS-0914986, and by
 the National Institute of General Medical Sciences under Grant
 NIH-1-R01-GM072000-01. Email: {\tt amos@cs.wisc.edu}}
\and  Zhiqiang Xu
\thanks{LSEC, Academy of Mathematics and Systems Sciences,
 Chinese Academy of Sciences, Beijing, 100190, China. The work of this
author was supported in part by NSFC grant 10871196 and was performed
in part at Technische Universit\"at Berlin. Email: {\tt xuzq@lsec.cc.ac.cn}} }

\date{\small February 10, 2010}

\maketitle

\begin{abstract}  Zonotopal algebra interweaves algebraic, geometric and
combinatorial properties of a given linear map $X$. Of basic significance in
this theory is the fact that the algebraic structures are derived from the
geometry (via a non-linear procedure known as ``the least map''), and that
the statistics of the algebraic structures (e.g., the Hilbert series of
various polynomial ideals) are combinatorial, i.e., computable using a
simple discrete algorithm known as ``the valuation function''. On the other
hand, the theory is somewhat rigid since it deals, for the given $X$, with
exactly two pairs each of which is made of a nested sequence of three ideals:
an external ideal (the smallest), a central ideal (the middle), and an
internal ideal (the largest).

In this paper we show that the fundamental principles of zonotopal
algebra as described in the previous paragraph extend far beyond
the setup of external, central and internal ideals by building a
whole hierarchy of new combinatorially defined zonotopal spaces.
\end{abstract}



%%%%%%%%%%%%%%%%%%%%%%%%%%%%%%%    INTRODUCTION    %%%%%%%%%%%%%%%%%%%%%%%%%%
\section{\label{sec:intro}Introduction}
%%%%%%%%%%%%%%%%%%%%%%%%%%%%%%%    MOTIVATION   %%%%%%%%%%%%%%%%%%%%%%%%%%%%%
\subsection{\label{sec:motiv}Motivation}
%%%%%%%%%%%%%%%%%%%%%%%%%%%%%%%%%%%%%%%%%%%%%%%%%%%%%%%%%%%%%%%%%%%%%%%%%%%%%

We are interested in this article in the study of algebraic structures,
most notably in terms of homogeneous zero-dimensional polynomial ideals,
over hyperplane arrangements, and, by duality, over zonotopes. We start
by describing the pertinent setup.

Let $X$ be a $n\times N$ matrix of full rank $n(\le N)$ all whose columns
are nonzero, which can be also viewed as a multiset of its columns. 
The theory presented in this paper is invariant under the order of the 
columns. In fact, some of the algorithms below will require us to order 
the columns,  but will produce output that is independent of the order. 
The {\bf zonotope} $Z(X)$ associated with $X$ is the polytope
  $$
  Z(X):=\{\sum_{x\in X} t_xx:\,\, t\in [0,1]^N\}.
  $$
In other words, the zonotope $Z(X)$ is the image of the unit cube
$[0,1]^N$ under $X$ viewed as a linear map from $\R^N$ to $\R^n$
or, equivalently, $Z(X)$ is a Minkowski sum of segments $[0,x]$
over all column vectors $x\in X$.

The theory of {\bf Zonotopal Algebra} (see \cite{HR}) is built
around three pairs of zero-dimensional homogeneous polynomial ideals
that are associated with $X$: an external pair $({\CI}_+(X),$ $
{\CJ}_+(X))$, a central pair $({\CI}(X), \CJ(X))$, and an internal
pair $(\CI_-(X), \CJ_-(X))$. These polynomial ideals play a role in
several different areas of mathematics. In  Approximation Theory,
these ideals provide important information about multivariate splines
on regular grids (box splines, see \cite{BHR93}). In Algebra, these ideals
appear, for example,  in the context of group representations and
also in the context of particular types of orthogonal polynomials.
In Combinatorics, these ideals are pertinent to algebraic graph
theory, and are intimately connected with the Tutte polynomial. The
list goes on, with the most direct connection being  particular topics
within convex geometry such as zonotopes, zonotope tilings, lattice
points in zonotopes, and various aspects of hyperplane arrangements.
For additional connections, see \cite{AP,HRcomb,SX,B09}.

A centerpiece in the theory of Zonotopal Algebra are the formul\ae\ that
capture the codimensions of the above-mentioned ideals, and more generally,
their Hilbert series in terms of the basic statistics of an associated
matroid. The other pillar of the theory is the connection between the
aforementioned ideals and various relevant geometric structures, viz., 
the zonotope $Z(X)$, and the hyperplane arrangement $\CH(X,\lam)$.

We pause in order to illustrate some of these aspects.
First, we let
 $$
 \Pi\,:=\,\C[t_1,\ldots,t_n]
 $$
be the space of polynomials in $n$ variables, and let
 $$
 \Pi_k^0
 $$
be the subspace of $\Pi$ that contains all
homogeneous polynomials of exact degree $k$.
Also, for any homogeneous ideal $I\subset \Pi$, we denote
$$
\ker I:=\{q\in \Pi : \;\; p(D)q=0 \;\;\; \hbox{\rm for all }\;  p\in
I\}.
$$
Here, $p(D)$ is the counterpart of $q\in \Pi$ in the ring
$\C[\partial/\partial t_1,\ldots,\partial/\partial t_n].$ Given a
zero-dimensional polynomial ideal $I\subset\Pi$, we denote by
$$\codim I$$
the dimension of the quotient space
$$\Pi/I,$$
which (is always finite and) is equal to $\dim \ker I$.

We next associate the matrix $X$ with a suitable {\bf hyperplane
arrangement} $\CH$: we consider each column $x\in X$ as a linear
functional $p_x$ in $(\Rn)^\ast$ (using the standard inner product
in $\Rn$), and denote by $H_{x,\lam_x}$ the zero set of the affine
polynomial
$$\Rn\ni t\mapsto p_x(t)-\lam_x,$$
with $\lam_x\in\R$ (arbitrary but fixed). The hyperplane
arrangement $\CH(X,\lam)$ is the union of the hyperplanes
$H_{x,\lam_x}$, $x\in X$. We assume that $\CH(X,\lam)$ is {\bf
simple } which means that every subcollection of $m$ hyperplanes has
either an empty intersection or an intersection of codimension $m$.
Note that, for any given $X$, all vectors $\lam\in\R^X$ for which
$\CH(X,\lam)$ is simple form an open dense subset of $\R^X$,
\cite{Ro88}.

Next, we describe the three $\CI$ ideals. To this end, we consider
submatrices $Y\subset X$ (that are obtained from $X$ by removing
some of its columns) that are of rank $n-1$. The column span of such
a submatrix is a {\bf facet hyperplane} of $X$, and we denote by
$$\CF(X)$$
the set of all facet hyperplanes of $X$. The normal  to a facet hyperplane
$F$ (defined uniquely  up to a scalar) is denoted by
$$\eta_F,$$
and the {\bf multiplicity} (in $X$) of a facet hyperplane $F\in \CF(X)$ is
the cardinality
$$m(F):=m_X(F):=\#\{x\in X: x\not\in F\}.$$
The three $\CI$ ideals are generated, each, by suitable powers
of the normal polynomials:
$$\{p_{\eta_F}^{m(F)+\eps}:\ F\in \CF(X)\}.$$
The internal ideal $\CI_-(X)$ corresponds to the choice $\eps=-1$,
the central ideal $\CI(X)$ corresponds to the choice $\eps=0$,
while the external ideal  $\CI_+(X)$ corresponds to the choice
$\eps=+1$. We now state a result that connects these three ideals
to the hyperplane arrangement $\CH(X,\lam)$. In the sequel, we
will also show the connection of these three ideals to the
zonotope $Z(X)$, and to several other constructs.

\begin{theorem}
\hfill
\begin{enumerate}
\item $\codim\CI(X)$ equals the number of vertices in the simple arrangement
$\CH(X,\lam)$.

\item $\codim\CI_+(X)$ equals the number of connected components
in $\Rn\bks \CH (X,\lam)$.

\item $\codim\CI_-(X)$ equals the number of bounded connected components
in $\Rn\bks \CH(X,\lam)$.

\end{enumerate}

\end{theorem}

An important highlight of the $\CI$-ideals is that their associated
kernels can be described cleanly and explicitly in terms of the
columns of $X$. Here, we discuss this point in the context of
the central zonotopal space.
Given $Y\subset X$,\footnote{Recall that we refer to $X$ as the multiset of its
columns, hence $Y$ is obtained by removing some columns of $X$.}
we say that $Y$ is {\bf short} if $\rank(X\bks Y)=n$. A {\bf short
polynomial} is a product
\begin{equation}
p_Y:=\prod_{y\in Y}p_y,
\end{equation}
over a short subset $Y$.
We let
$$
\CP(X) := \spam\{p_Y:\rank(X\setminus Y)=n\},
$$
be the span of the short polynomials.
A subset $Y\subset X$ that is not short is called {\bf long}.
Let $\CJ(X)$ denote the ideal generated by the {\bf long
polynomials}:
$$
\CJ(X):={\rm Ideal}\{p_Y: Y\subset X,\ \rank(X\bks Y)<n \},
$$
 and set
$$
\CD(X) := \ker \CJ(X).
$$
Below we collect some  of the main results of
\cite{AS88,BDR91,BR91,DM85,DM89,DR90}; see
also~\cite[Theorem~3.8]{HR}, where this summary appears in its
present form. We use the notation
$$
\B(X)\,:=\,\{B\subset X : B \mbox{ is a basis for } \R^n\},
$$
as well as the pairing
$$\Pi\to\Pi^\ast: p\mapsto \inpro{p,\cdot},\quad
\inpro{p,q}:=(p(D)q)(0)=(q(D)p)(0).$$
The space $\Pi(V)$ that appears in Theorem~\ref{th:exzono}
is defined in Section~\ref{sec:least}.

\goodbreak
\begin{theorem}\label{th:exzono}
\hfill
\begin{enumerate}
\item $\dim\CP(X)=\dim\CD(X)=\#\B(X)$.
\item The map $p\mapsto\inpro{p,\cdot}$ is an isomorphism
between $\CP(X)$ and $\CD(X)^\ast$.
\item $\CD(X)=\Pi(V)$, with $V$ the vertex set of the hyperplane arrangement
$\CH(X,\lam)$.
\item $\CP(X)=\ker\CI(X)$.
\item $\CP(X)\bigoplus\CJ(X)=\Pi$.

\end{enumerate}
\end{theorem}

%%%%%%%%%%%%%%%%%%%     THE LEAST MAP   %%%%%%%%%%%%%%%%%%%%%%%%%%%%%%%%%%%
\subsection{\label{sec:least}The least map $V\mapsto \Pi(V)$}
%%%%%%%%%%%%%%%%%%%%%%%%%%%%%%%%%%%%%%%%%%%%%%%%%%%%%%%%%%%%%%%%%%%%%%%%%%%

\begin{definition} Let $V$ be a finite pointset in $\Rn$. Given
$v\in \Rn$, let
$$e_v: t\mapsto e^{v\cdot t}$$
be the exponential with frequency $v$, and define
$$\Exp(V):=\spam\{e_v:\ v\in V\}.$$
Given $f\in \Exp(V)$, let
$$f=\sum_{j=j_f}^\infty f_j,\quad j_f\ge 0,$$
be its homogeneous power expansion, i.e., $f_j\in \Pi_j^0$, for all $j$,
and $f_{j_f}\not=0$.
Define
$$f\least:=f_{j_f},\quad \Pi(V):=\spam\{f\least: f\in \Exp(V)\bks 0\}.$$
\end{definition}

Our interest in this paper is focused on point sets $V$ that are
either subsets of the set of integer points in the zonotope
$Z(X)$, or subsets of the vertex set $V(X,\lam)$ of the hyperplane
arrangement $\CH(X,\lam)$. Note that in the latter case, since we
assume $\CH(X,\lam)$ to be simple, we have that
$\#V(X,\lam)=\#\B(X)$, and there is a set bijection
$$V\ni v\mapsto B_v\in \B(X),$$
that sends each vertex $v$ to the set of columns of $X$ whose
hyperplanes contain $v$. Thus, simplicity means in this context
that the latter set is always a basis, and never larger than that.
Now, given $V'\subset V(X,\lam)$, let $\B':=\B'(X):=\{B_v: \ v\in
V'\}$, and define the ideal
$$\CJ_{\B'}(X):=\Ideal\{p_Y: \;\;  Y\subset X, \;\;  Y\cap B\not=\emptyset,
\text{ for all } B\in\B'\}.$$

\begin{theorem}[\cite{BR90,BR91}]\label{th:pi}

\hfill
\begin{enumerate}
\item For every finite $V\subset \Rn$, the restriction map
$$\Pi(V)\ni f\mapsto f_{|V}$$ is an isomorphism between $\Pi(V)$
and $\C^V$. In particular, $\dim \Pi(V)=\#V$.

\item
With $V'$, $\B'$ and $\CJ_{\B'}(X)$ as above, we have that
$\Pi(V')\subset \ker \CJ_{\B'}(X)$. In particular,
$$\codim\CJ_{\B'}(X)\ge \dim\Pi(V')=\#V'=\#\B'.$$
\end{enumerate}
\end{theorem}

Note that the choice $V':=V(X,\lam)$ leads to $\B'=\B(X)$, and to
$\CJ_{\B'}(X)=\CJ(X)$. Hence the above result shows that
$\Pi(V(X,\lam))\subset \ker\CJ(X)$.  Theorem~\ref{th:exzono} asserts
that for this particular choice of $V'$ equality holds:
$\Pi(V(X,\lam))= \ker\CJ(X)$. For other choices of $\B'$, however, 
this inclusion may be proper.

\smallskip
The least map, thus, connects  the vertices of $\CH(X,\lam)$ to
the $\CJ$-ideals. It also connects the integers points in the 
zonotope $Z(X)$ to the $\CI$-ideal. To this end, recall that
$X$ is {\bf unimodular} if $X\subset \Zn$ and $|\det B|=1$, for every
$B\in \B(X)$. Furthermore, let $\I(X)$  be the collection of
(linearly) {\bf independent subsets} of $X$:
$$
\I(X)\,:=\,\{I\subset X: I \mbox{ is independent in } \R^n\}.
$$

\begin{theorem}[\cite{HR}]\label{th:plus}

\hfill
\begin{enumerate}
\item $\ker \CI_+(X)=\CP_+(X):=\spam\{p_Y:\ Y\subset X \}$.
\item $\dim\CP_+(X)=\# \I(X)$.
\item Assume $X$ is unimodular. Then $\CP_+(X)=\Pi(Z(X)\cap \Zn)$.
\end{enumerate}
\end{theorem}

%%%%%%%%%%%%%%%%%%%  HILBERT FUNCTIONS  %%%%%%%%%%%%%%%%%%%%%%%%%%%%%%%%%
\subsection{Hilbert series}
%%%%%%%%%%%%%%%%%%%%%%%%%%%%%%%%%%%%%%%%%%%%%%%%%%%%%%%%%%%%%%%%%%%%%%%%%

The Hilbert series of the three $\CI$-ideals are closely related to
the external activity of the Tutte polynomial of the (matroid of the) 
given multiset $X$. Let $\prec$ be any order on $X$. Given a set $Y\subset X$, 
we define the valuation of $Y$ (per the given order) by
$$\val(Y):=\#X(Y),\quad X(Y):=\{x\in X\bks Y:\ x\not\in \spam\{y\in Y:\
y\prec x\}\}.$$  This valuation function determines the Hilbert
series of the $\CI$-ideals as follows.

\begin{theorem}[\cite{DR90,HR}]\label{th:basis}

\hfill
\begin{enumerate}
\item The polynomials
$$Q_B:=p_{X(B)},\quad B\in\B(X)$$
form a basis for $\CP(X)$. In particular, for every positive integer $j$,
$$\dim(\CP(X)\cap \Pi_j^0)=\#\{B\in \B(X) : \; \val(B)=j\}.$$
\item The polynomials
$$Q_I:=p_{X(I)},\quad I\in\I(X)$$
form a basis for $\CP_+(X)$. In particular, for every positive integer $j$,
$$\dim(\CP_+(X)\cap \Pi_j^0)=\#\{I\in \I(X) : \; \val(I)=j\}.$$
\end{enumerate}
\end{theorem}

%%%%%%%%%%%%%%%%%%   INTERMEDIATE SETUPS  %%%%%%%%%%%%%%%%%%%%%%%%%%%%%%%%%
\subsection{Intermediate setups}
%%%%%%%%%%%%%%%%%%%%%%%%%%%%%%%%%%%%%%%%%%%%%%%%%%%%%%%%%%%%%%%%%%%%%%%%%%%

Zonotopal algebra interweaves, thus, algebraic, geometric and combinatorial
properties of the linear map $X$. Of basic significance to this theory is the
fact that the algebraic structures are derived from the geometry (via
the least map), and that the statistics of the algebraic structures
(i.e., the various Hilbert series) are combinatorial, i.e.,
computable using the valuation function. On the other hand, the theory is somewhat
rigid since it deals with exactly three sets of ideals for each given $X$,
their three Hilbert series and so on.

In this paper we show that the fundamental principles of zonotopal
algebra as described in the previous paragraph extend far beyond
the previously known setups of external, central and internal ideals. 
For example, let $\I'\subset \I(X)$, and let us define $Y\subset X$ 
to be {\bf $\I'$-short} if $X\bks Y$ contains an element of $\I'$. Define
$$\CP_+(X,\I')$$
to be the span of all $\I'$-short polynomials. The central case
$\CP(X)$ corresponds to the choice $\I'=\B(X)$, and the external
case $\CP_+(X)$ corresponds to the choice $\I'=\I(X)$. In either
of these extreme cases, the dimension of the polynomial space
$\CP_+(X,\I')$ coincides  with the cardinality of $\I'$.

The following questions now arise naturally:

\begin{itemize}
\item Do we have $\dim\CP_+(X,\I')=\#\I'$?
\item Is the ideal $\CI_+(X,\I')$ of all differential operators
that annihilate  $\CP_+(X,\I')$ still generated by powers
of the normals  $p_{\eta_F}$ to the facets?
\item Do the polynomials $Q_I$, $I\in \I'$ form a basis for
$\CP_+(X,\I')$ or, at least, does the valuation of the set
$X(I)$, $I\in\I'$ determine the Hilbert series of $\CP_+(X,\I')$?
\item Do we have a dual setup in terms of an ideal $\CJ_+(X,\I')$
that is generated by a suitable notion of long polynomials,
so that its kernel  is connected to a suitable set of vertices
of some hyperplane arrangement?
\end{itemize}

\medskip
It is somewhat surprising (at least to us) that the questions above
can all be answered in the affirmative for a large class of sets $\I'$.
We deal in  this paper with two different setups.

In the first one, which we refer to as semi-external, we select an
arbitrary subset $\I'$ of $\I(X)$ and impose only one condition on
it, viz., that $\I'$ should be closed under (subspace) inclusion in the
sense that if $\spam I\subset \spam I'$, for $I\in \I'$ and $I'\in
\I(X)$ then $I'\in\I'$. We introduce then a suitable \linebreak $\CJ$-ideal,
$\CJ_+(X,\I')$, and  a suitable $\CI$-ideal $\CI_+(X,\I')$, provide an
explicit description and bases for $\ker\CI_+(X,\I')$, develop an
algorithm for computing the Hilbert series of these ideals (both
share the same Hilbert series) in terms of the aforementioned
valuation, and describe a suitable geometric derivation of
$\ker\CJ_+(X,\I')$ (by acting on vertices of hyperplane
arrangements). The ideal $\CI_+(X,\I')$ is not generated, however,
by powers of the normal polynomials $p_{\eta_F}$, $F\in\CF(X)$.
That said, by assuming slightly more on $\I'$, that property can
be guaranteed as well.

In the second setup, which we refer to as semi-internal,
we select and fix an independent set $I\in \I(X)$, and develop a
theory where the focus is on the space
$$\cpm{X,I}:=\cap_{b\in I}\CP(X\bks b).$$
When $I$ is a basis, the above space coincides with the kernel
of $\CI_-(X)$. Once again, this case gives rise to a theory that parallels
in its ingredients the one outlined for the semi-external case.
In particular, we prove that the corresponding $\CI$-ideal is generated
by powers of the normals, and we further identify a subset
$\cbm{X,I}\subset \B(X)$ whose valuation determines the Hilbert
series of $\cpm{X,I}$.
\begin{remark}
In line with the tradition in multivariate spline theory,
we define our objects (polynomials, differential operators) on
the Euclidean space $\Rn$, using the natural inner product on that
space.  Thus, in our presentation $\Rn$ serves simultaneously as 
a vector space and as its dual space. Had we followed the tradition
in algebra and separated the space from its dual, some of the definitions 
and constructions would have been as follows: Start with an $n$-dimensional
vector space $W$. Then $X$ is a collection of nonzero vectors
in the dual vector space $W^\ast$. Each element of $X$ defines a
hyperplane in $W$. 
Let $t_1,\ldots,t_n$ be a basis of $W$, and let
$s_1,\ldots,s_n$ be the dual basis of $W^\ast$. The pairing that
we use $\inpro{p,q}=(p(D)q)(0)$ is in fact a pairing between
the polynomial rings
$\Pi=\C[t_1,\ldots, t_n] (= \Sym W^\ast)$ and $\Pi^\ast =\C[s_1,\ldots,s_n] ( = \Sym W)$. 
The $\CP$-spaces and the $\CJ$-ideals are subspaces of\, $\Pi$.
The $\CD$-spaces and the $\CI$-ideals are subspaces of\, $\Pi^\ast$.
The hyperplane arrangement $\CH(X)$ lies in $W$, while the zonotope
$Z(X)$ lies in $W^\ast$. 
\end{remark}

%%%%%%%%%%%%%%%%  SEMI_EXTERNAL SPACES  %%%%%%%%%%%%%%%%%%%%%%%%%%%%%%%%%%
\section{\label{sec:semi-ext}Semi-external zonotopal spaces}
%%%%%%%%%%%%%%%%%%%%%% EXTERNAL REVIEW %%%%%%%%%%%%%%%%%%%%%%%%%%%%%%%%%%%
\subsection{\label{sec:ext-review}A review of external zonotopal spaces}
%%%%%%%%%%%%%%%%%%%%%%%%%%%%%%%%%%%%%%%%%%%%%%%%%%%%%%%%%%%%%%%%%%%%%%%%%%

Recall that we consider our $n\times N$ matrix $X$ of full rank $n$ as a finite multiset
$X\subset \R^n\setminus \{0\}$ of size $N=\#X$. Also recall from Section~\ref{sec:motiv}
that $\B(X)$ denotes the multiset of all (linear) bases of $X$, and $\I(X)$
denotes the multiset of all (linearly) independent subsets of $X$.

The definition of the external ideal $\CJ_+(X)$ requires to choose
an additional basis $B_0$ for $\Rn$, and to order its elements.
There is no restriction on the choice of $B_0$ or on the chosen order
 $\prec$, but the definition of the ideal $\CJ_+(X)$ depends on the
choice of $B_0$ and the choice of the order. We augment $X$ by $B_0$
and define $X':=X\sqcup B_0$, where $\sqcup$ denote the union of two
 multisets, i.e., a collection obtained by listing all vectors in $X$
and all vectors in $B_0$.
%olga
We extend the order $\prec$  to a full order on $X'$, with the only
requirement on this extension that $x\prec b$,
for every $x\in X$ and $b\in B_0$. We use $B_0$ to extend each
independent subset $I\in \I(X)$ to a basis ${\rm ex}(I)\in \B(X')$
by a greedy completion,  i.e.,
$b\in \ex(I)$ if and only if $b\in I$ or else $b\in B_0$ and
$$
b\notin \spam \{I\cup \{b'\in B_0:b'\prec b\}\}.
$$
For each $I\in \I(X)$, we define
$$
X(I)\,:=\, \{x\in X\bks I: x\notin \spam \{b\in I: b\prec x\}\}.
$$
We associate each independent subset $I$ with the polynomial
$$Q_I:=p_{X(I)}.$$
Before proceeding  with introduction and analysis of the
semi-external zonotopal spaces, we pause momentarily in order to
discuss the (full) external case, whose theory was developed in
\cite{HR}. To this end, we recall the definition of the ideal
$\CI_+(X)$ and the space $\CP_+(X)$
from the introduction, and add the following definitions:

\begin{definition}
\begin{eqnarray*}
 \CJ_+(X)\,&:=&\,\Ideal\{p_Y:
Y\subset X',\ Y\cap \ex(I)\not=\emptyset,\
\text{ for all } I\in \I(X)\},\\
\CD_+(X)\, &:=&\, \ker \CJ_+(X).
\end{eqnarray*}
\end{definition}

Let
$\CH(X',\lam)$ be a simple hyperplane arrangement associated with
the extended set $X'$. Let $V':=V(X')$ be the vertex set of this
arrangement. Since $\CH(X',\lam)$ is assumed to be simple, there
is a bijection  $$\Cb: \B(X')\to V'$$
in which each basis $B$ is mapped to the intersection of the hyperplanes
$\{H_{b,\lam_b}: b\in B\}$. We denote
$$V_+:=\{\Cb(\ex (I)): I\in \I(X)\}.$$
Note that $V_+$ depends on $B_0$, on the order $\prec$ imposed on $B_0$,
and the parameter vector $\lam$. (It does not depend however on the
way $\prec$ is extended to $X$.)

Before describing the main result from \cite{HR} regarding the full external
case, we  recall the pairing from Section~\ref{sec:motiv} that plays an
important role in the underlying duality between the spaces $\CP_+(X)$ and
$\CD_+(X)$: Given two polynomials $p$ and $q$, the pairing $\langle p , q \rangle$
is defined  by
$$  \inpro{p,q} := (p(D)q)(0). $$
We now describe the pertinent result from \cite{HR}:

\goodbreak
\begin{theorem}[{\cite[Theorem~4.10]{HR}}]\label{th:explus}
\hfill
\begin{enumerate}
\item $\dim\CP_+(X)=\dim\CD_+(X)=\#\I(X)$.
\item $\CD_+(X)=\Pi(V_+)$, with $V_+$ as above.
\item $\CJ_+(X)\oplus\CP_+(X)=\Pi$.
\item $\CP_+(X)=\ker\CI_+(X)$.
\item The pairing
$\inpro{\cdot,\cdot}$ defines an isomorphism  between $\CP_+(X)$ and
the dual $\CD_+(X)^\ast$ of $\CD_+(X)$.
\item The polynomials $\{Q_I: I\in \I(X)\}$ (that depend on the
order $\prec$) form a basis for $\CP_+(X)$ (whose definition is
independent of that order). In particular, the Hilbert series of
$\CP_+(X)$ is determined by  $$h_+(j):=h_{+,X}(j):=\#\{I\in \I(X):
\val(I)=j\}.$$
\end{enumerate}

\end{theorem}

Under an additional assumption that $X$ is unimodular, the results in
\cite{HR} also draw  a connection between the integer points
$\CZ(X)$ in the closed zonotope $Z(X)$, and the external zonotopal
spaces, viz.\,
$$\CP_+(X)\,\,=\,\,\Pi(\CZ(X)).$$

%%%%%%%%%%%%%%%%%%%%%%%%%%  SEMI-EXTERNAL ANALYSIS  **********************
\subsection{\label{sec:analysis-ext}Introduction and analysis of semi-external zonotopal spaces}
%%%%%%%%%%%%%%%%%%%%%%%%%%%%%%%%%%%%%%%%%%%%%%%%%%%%%%%%%%%%%%%%%%%%%%%%%%

The central zonotopal spaces  are defined with respect to the set of
all bases $\B(X)$. The external ones are defined with respect to the
independent sets $\I(X)$. The  semi-external spaces we now introduce
are defined by selecting  a set $\I'$ in between: $$\B(X)\subset
\I' \subset \I(X).$$

\begin{definition}
Let $\I'\subset \I(X)$. We say that $\I'$ is {\bf semi-external} if
the following holds:
$$\hbox{If $I\in \I'$, $I'\in\I(X)$  and $\spam I\subset \spam I'$,
then $I'\in \I'$.}$$
\end{definition}
Note that a (non-empty) semi-external set $\I'$ must contain $\B(X)$.

Our goal is to define and analyse the zonotopal ideals and the zonotopal spaces
that are associated with a semi-external collection $\I'$. To this end, we denote
$$S(X,\I'):=\{\spam I: I\in \I(X)\setminus \I'\},$$
and define

\begin{definition}
\begin{eqnarray*}
 \CI_+(X,\I')\,&:=&\,\CI_+(X)+\Ideal\{\Pi^0_{\#(X\setminus
S)}(S{\perp}): S\in S(X,\I')\},\\
 \CJ_+(X,\I')\,&:=&\,\Ideal\{p_Y:
Y\subset X',\ Y\cap \ex(I)\not=\emptyset,\
\forall I\in \I' \},\\
\CP_+(X,\I')\,&:=&\,\spam\{p_Y:\ Y\subset X, \, Y\cap I=\emptyset\
\hbox{for some $I\in \I'$}\},\\
\CD_+(X,\I')\, &:=&\, \ker \CJ_+(X,\I').
\end{eqnarray*}
\end{definition}

\noindent To be sure, the polynomial space $\Pi^0_j(S{\perp})$
consists of homogeneous polynomials of degree $j$ on the orthogonal
complement  of $S$ in $\Rn$. Equivalently,
$\Pi(S{\perp}):=\{p\in\Pi\; : \;D_\eta p=0 \;\; \hbox{\rm for all}
\;\; \eta\in S\}$.
%amos

\smallskip
We note that the semi-external spaces capture the external ones and
the central ones as special cases. In the case $\I'=\I(X)$, the
$\I'$-ideals $(\CI_+(X,\I'),\CJ_+(X,\I'))$ coincide with the external
ideals. In the case $\I'=\B(X)$, they coincide with the central ideals.

Here is a simple example showing how the  space $\CP_+(X,\I')$ and
the ideal $\CI_+(X,\I')$ change depending on the choice of our
special collection $\I'$ of independent subsets.
%Throughout the
%example, $e_j$ denotes the $j$th standard  basis vector in $\R^n$
%($n=3$).

\begin{example}
$$ X \quad = \quad
\begin{bmatrix}
\,1 \,& \,0 \,& \,0\, & 1 \, \\
\,0 \,& \,1 \,& \,0 \,&  1\,\\
\,0\, & \,0 \,& \,1 \,& 1 \, \\
\end{bmatrix}=:[x_1,x_2,x_3,x_4]
$$
If
$$
\I' \,=\,\{[x_1,x_2],[x_1,x_3],[x_1,x_4]\}\cup \B(X),
$$
then
\begin{eqnarray*}
\CP_+(X,\I') &\,=\,& \spam\{1,p_{x_1},p_{x_2},p_{x_3},p_{x_2}p_{x_3},
p_{x_3}p_{x_4},p_{x_2} p_{x_4}  \},\\
\CI_+(X,\I') &\,=\,&
\Ideal\{p_{x_3}^3,p_{x_2}^3,p_{x_2-x_3}^3,p_{x_1}^2,
p_{x_1-x_3}^2,p_{x_1-x_2}^2,p_{x_3}^2 p_{x_2} \}.
\end{eqnarray*}
 If
$$
\I' \,=\,\{[x_1], [x_1,x_2],[x_1,x_3],[x_1,x_4]\}\cup \B(X),
$$
then
\begin{eqnarray*}
\CP_+(X,\I') &\,=\,& \spam\{1,p_{x_1},p_{x_2},p_{x_3},p_{x_2}p_{x_3},
p_{x_3}p_{x_4},p_{x_2} p_{x_4}, p_{x_2} p_{x_3} p_{x_4}  \},  \\
\CI_+(X,\I') &\,=\,&
\Ideal\{p_{x_3}^3,p_{x_2}^3,p_{x_2-x_3}^3,p_{x_1}^2,
p_{x_1-x_3}^2,p_{x_1-x_2}^2  \}.
\end{eqnarray*}
\end{example}

\begin{theorem}\label{thm 2.6}
\hfill
\begin{description}
\item{1.} The polynomials $(Q_I)_{I\in \I'}$ form a basis for $\CP_+(X,\I')$.
In particular, $\dim \CP_+(X,\I')=\# \I'$. The Hilbert series of
 $\CP_+(X,\I')$ is determined by these polynomials:
$$h_{+,\I'}(j):=\dim(\CP_+(X,\I')\cap \Pi_j^0)=\#\{I\in \I': \val(I)=j\}.$$

\item{2.} Let $V(\I')$ be the vertices in the arrangement $\CH(X',\lam)$
that correspond to $\ex(\I')$. Then
$\CD_+(X,\I')=\Pi(V(\I')).$ In particular, $\dim\CD_+(X,\I')=\#\I'$.

\item{3.} $\CJ_+(X,\I')\oplus\CP_+(X,\I')=\Pi$.

\item{4.} $\CP_+(X,\I')=\ker\CI_+(X,\I')$.

\item{5.} The map $p\mapsto \inpro{p, \cdot}$ is an isomorphism between
$\CP_+(X,\I')$ and $\CD_+(X,\I')^\ast$.
\end{description}
\end{theorem}
\begin{proof}
1.\/ We first prove that $\{Q_I:I\in \I'\}$ is a basis for
$\CP_+(X,\I')$.

The fact that $Q_I\in \CP_+(X,\I')$ for every $I\in \I'$ is trivial since
$Q_I=p_{X(I)}$, with $X(I)\cap I=\emptyset$. Now, we choose $I'\in \I'$,
and $Y\subset X\setminus I'$. According to the definition  of $\I'$, since
$\spam I'\subset S:=\spam (X\setminus Y)$, any basis
$I\subset X\setminus Y$ for $S$ must lie in $\I'$. We set
$\tilde{X}:=X\cap S$. Now, with $\{ \tilde{Q}_B:=p_{\tilde{X}(B)}:
B\in \B(\tilde{X})\}$ the homogeneous basis for $\CP(\tilde{X})$
(per the fixed order we chose for $X$) \cite{DR90},
we have
\begin{eqnarray*}
p_Y\in p_{X\setminus S}\cdot \CP(\tilde{X})&=&p_{X\setminus
S}\cdot \spam\{\tilde{Q}_B:
B\in \B(\tilde{X})\}\\
& =&\spam\{Q_{I}: I\in \B(\tilde{X})\}\subset \spam\{Q_I: I\in \I' \},
\end{eqnarray*}
with the last inclusion following from the fact that every basis for
$S$ from $X$ must lie in $\I'$,  since $\I'$ is semi-external. Thus,
$$\CP_+(X,\I')\subset \spam\{Q_I: I\in \I' \}.$$
We conclude that the two spaces coincide and
$$\dim\CP_+(X,\I')=\#\I'.$$

\smallskip
\noindent 2. and 3. : \/  Note that $\{\ex(I):I\in \I' \}\subset
\B(X')$. We first apply Theorem \ref{th:pi} to this case (i.e.,
with $X$ there replaced by $X'$ and $\B'$ being our $\{\ex(I):I\in
\I' \}$). The theorem thus tells us that
$$
\Pi(\CV(\ex(\I')))\subset \CD_+(X,\I'),
$$
and that $\dim\CD_+(X,\I')\geq \# \I'$.
%amos

We next claim that
$\CP_+(X,\I')+\CJ_+(X,\I')=\Pi$.  Once we prove this claim, we will have that
$\dim\CD_+(X,\I')=\codim\CJ_+(X,\I')\le\dim\CP_+(X,\I')=\#\I'$. This will yield that
$\dim\CD_+(X,\I')=\#\I'$, hence that
$\CD_+(X,\I')=\Pi(V(\I'))$. The same will also yield that $\codim\CJ_+(X,\I')=\#\I'$,
hence that the sum $\CP_+(X,\I')+\CJ_+(X,\I')$ is direct (since $\dim\CP_+(X,\I')=\#\I'$,
too). In summary, the proofs of (2) and (3) will be completed once we show
that $\CP_+(X,\I')+\CJ_+(X,\I')=\Pi$, as we do now.

Since we know that $\CP(X)\oplus \CJ(X)=\Pi$ (see \cite{HR}; the
result was first proved in \cite{AS88,DR90}), we conclude from the
fact that $\CP(X)\subset \CP_+(X,\I')$, that
$\CP_+(X,\I')+\CJ(X)=\Pi$. We prove now that
$$\CJ(X)\, \subset\, \CP_+(X,\I')+\CJ_+(X,\I').$$
A polynomial in $\CJ(X)$ is a linear combination of polynomials of
the form $p_Y f$,  where $f\in\Pi$, $Y\subset X$ and
$\rank(X\setminus Y)<n$. Let $I\subset X\setminus Y$ be a basis for
$\spam X\setminus Y$. If $I\not\in \I'$, then, since $\I'$ is semi-external,
$X\setminus Y$ contains no element  in $\I'$, hence $p_Y\in
\CJ_+(X,\I')$, {\it a fortiori\/} $p_Yf\in \CJ_+(X,\I')$, and we are done.
Otherwise, $I\in \I'$. We prove this case by induction on $\#(X\bks Y)$.
Thus, we assume that $p_{Y'}
f\in \CP_+(X,\I')+\CJ_+(X,\I')$ whenever $\#(X\setminus Y')=k\ge -1$ and consider
$p_Y f$ where $\#(X\setminus Y)=k+1$ (the initial case $X=Y$ corresponds
to the choice $k=-1$).

The proof goes as follows:
we denote $B:=\ex(I)$, and, using the fact that
$$\Pi=\Pi_0^0+\Ideal\{p_b: b\in B\},$$ write
$$
f=c+\sum_{b\in B}c_b p_b f_b,
$$
with $c$ and $(c_b)_{b\in B}$ some scalars, and $(f_b)_{b\in B}$
some polynomials. Then
$$p_Yf=cp_Y+\sum_{b\in B}c_b p_{Y\sqcup b}f_b.$$
We show that each of the terms on the right-hand-side lies in
$\CP_+(X,\I')+\CJ_+(X,\I')$.

Starting with $p_Y$, we note that $Y\cap I=\emptyset$ and that $I\in \I'$,
hence that $p_Y\in \CP_+(X,\I')$.  We then consider the term $p_{Y\sqcup b}f_b$
under the assumption that $b\in I$. In this case,
$Y':=Y\sqcup b\subset X$ and $\#(X\setminus Y')=k$, and hence the induction
hypothesis applies to yield that
$p_{Y\sqcup b}f_b\in \CP_+(X,\I')+\CJ_+(X,\I')$.

Finally, we treat the summand $p_{Y\sqcup b}f_b$  under the assumption
that $b\in B_0$. Let
$I'\in \I(X)$, and assume that $Y\cap I'=\emptyset$. Then $I'\subset X\bks
Y$, hence $\spam I'\subset \spam I$, and therefore $\ex(I)\cap B_0\subset \ex(I')
\cap B_0$. Consequently, $b\in \ex(I')$. In summary, the set $Y\sqcup b$
intersects every extended basis $\ex(I')$, $I'\in \I'$, and this means that
$p_{Y\sqcup b}f_b$ lies in the  ideal $\CJ_+(X,\I')$.

\smallskip
\noindent 4. \/ We now prove that $\CP_+(X,\I')=\ker
\CI_+(X,\I')$. We first show that $\CP_+(X,\I')\subset \ker
\CI_+(X,\I')$. To this end, choose $Y\subset X$ such that
$X\setminus Y$ contains a set $I\in \I'$. We need to show that
$p_Y$ is annihilated by each of the generators of the ideal
$\CI_+(X,\I')$. The fact that $p_Y$ is annihilated by $\CI_+(X)$
follows from the fact that
$\ker\CI_+(X)=\CP_+(X)\supset\CP_+(X,\I')$ by
Theorem~\ref{th:explus}.

We now deal with a differential operator $q(D)$, $q\in \Pi^0_{\#(X\setminus
S)}(S{{\perp}})$, $S\in S(X,\I')$. Note that
$$q(D)p_Y=p_{Y\cap S}\;q(D)p_{Y\setminus S}.$$
Now, $\#(X\setminus S)=\deg q$, and $Y\bks S\subset X\bks S$. Therefore,
we just need to rule out the possibility that
$Y\setminus S= X\setminus S$. Indeed, if $Y\setminus S= X\setminus S$ then
$X\setminus Y\subset S$, and hence $I\subset S$. Since $I\in \I'$ and $\I'$ is
semi-external, this implies that every basis for $S$ from $X$ is in $\I'$,
contradicting thereby the assumption that $S\in S(X,\I')$.
Hence $\#(Y\setminus S)<\deg q$, and we obtain that $q(D)p_Y=0$.
Consequently, $\CP_+(X,\I')\subset\ker \CI_+(X,\I')$.

\smallskip
Proving the converse inclusion is somewhat harder. First, since
$\CI_+(X)\subset \CI_+(X,\I')$ directly from the definition of
$\CI_+(X,\I')$, we have that $\ker\CI_+(X,\I')\subset \ker
\CI_+(X)=\CP_+(X)$ (cf. Theorem~\ref{th:explus}). In addition, since
$(Q_I)_{I\in \I'}$ is a basis for $\CP_+(X,\I')$, while $(Q_I)_{I\in
\I(X)}$ is a basis for $\CP_+(X)$, we have that
$$\CP_+(X,\I')\subset \ker\CI_+(X,\I')\subset \ker \CI_+(X)=\CP_+(X)=\CP_+(X,\I')+\Q_{\I'},$$
with
$$\Q_{\I'}:=\spam\{Q_I: I\in \I(X)\setminus \I'\}.$$
We show now that $\Q_{\I'}\cap\ker\CI_+(X,\I')=0$, and this will complete
the proof.

To this end, let
$$
 f:=\sum_{I\in \I(X)\setminus \I'}c(I)Q_I\in \ker \CI_+(X,\I').
$$
We will show that $c(I)=0$ for all $I\in \I(X)\setminus \I'$. Assume,
to the contrary, that $c(I')\not=0$ for some $I'\in \I(X)\setminus \I'$,
and assume, without loss of generality, that $c(I)=0$ for every
$I\in \I(X)\setminus \I'$ of smaller cardinality. $I'$ cannot be a basis
in $\B(X)$, since it is not in $\I'$. We therefore extend $I'$ to a
basis $B'\in \B(X)$ using the vectors $(b_1,\ldots,b_k)\subset X$, $k\ge 1$,
and denote
$$S_0:=\spam\, I',\quad S_i:=\spam \left(I'\cup\{b_1,\ldots,b_i\}\right),
\quad i=1,\ldots, k.$$
%amos
For each $1\le i\le k$, we choose  vector $0\not=\eta_i\in S_i$, such
that $\eta_i\perp S_{i-1}\supset S_0$. Setting $m_i:=\#((X\cap
S_i)\setminus S_{i-1})$, $i=1,\ldots,k$, we define
$$
q\,:=\,\prod_{i=1}^kp_{\eta_i}^{m_i}.
$$
Since $\eta_i\perp S_0$ for every $i$, and since $\sum_{i=1}^k
m_i=\#(X\setminus S_0)$, we conclude that $q\in \Pi^0_{\#(X\setminus
S_0)}(S_0{\perp})$, which implies that $q\in \CI_+(X,\I')$ and that
$q(D)f=0$.

Now, for $Y\subset X$,
$$
q(D)p_Y\,=\,
p_{\eta_{1}}^{m_{1}}(D)p_{(Y\cap S_{1})\bks S_0}
p_{\eta_{2}}^{m_{2}}(D)p_{(Y\cap S_{2})\bks S_{1}}
\ldots
p_{\eta_{k-1}}^{m_{k-1}}(D)p_{(Y\cap S_{k-1})\bks S_{k-2}}
p_{\eta_k}^{m_k}(D)p_{Y\setminus S_{k-1}}.
$$
Since $\#(X\cap S_j)\bks S_{j-1}=m_j$, $1\le j\le k$,
we have $q(D)Q_I\neq 0$ only if $X\setminus S_{0}=X(I)\bks S_{0}$,
which is possible only if $I\subset S_{0}$ (since $I\cap X(I)=\emptyset$),
in which case
$$
q(D)Q_I=ap_{X(I)\cap S_0},
$$
for some $a\not=0$. Thus, with
%amos
$$
{\rm Sub}(S_0):=\{I\in \I(X)\setminus \I':  I\subset S_0\},
$$
we have
$$
0=q(D)f=a\sum_{I\in {\rm Sub}(S_0)}c(I)p_{X(I)\cap S_0}.
$$
%amos
Per our assumption on the minimal cardinality of $I'$, we have that
$c(I)=0$ whenever $I\in {\rm Sub}(S_0)$ and $\#I<\#I'$. However, the polynomial
set $\{p_{X(I)\cap S_0}:I\in \I(X),  \spam I=S_0\}$ is a basis for
the central space $\CP(X\cap S_0)$, hence we conclude that $c(I)=0$,
for every $I\in {\rm Sub}(S_0)$. In particular, $c(I')=0$. \smallskip

\noindent 5. \/  Pick $q\in \CD_+(X,\I')\setminus\{ 0\}$. Since
$\CJ_+(X,\I')\oplus\CP_+(X,\I')=\Pi$, $q$ can be written in the
form of $q=f+p $ where $f\in \CJ_+(X,\I')$ and $p\in
\CP_+(X,\I')$. $\CD_+(X,\I')=\ker\CJ_+(X,\I')$ implies that
$\inpro{q,f}=0$. We conclude that $\inpro{q,p}=\inpro{q,q}\neq 0$,
since $q\neq 0$. This means that there exists no $q\in
\CD_+(X,\I')\setminus \{0\}$ that satisfies
$$
\inpro{q,p} =0 , \,\,\,\, \forall p\in \CP_+(X,\I').
$$
%amos
The result follows from the fact that $\dim\CP_+(X,\I')=\dim\CD_+(X,\I')$.
\end{proof}

\smallskip
One observes that the definition of the ideal $\CI_+(X,\I')$
involves more than the powers of the normals to the facet
hyperplanes. We now investigate a special situation where the set
$\I'$ satisfies an additional condition, and show that
$\CI_+(X,\I')$ is generated then by  powers of the normals to the
hyperplanes and by nothing else. Precisely, we will assume that an
independent subset $I_0$ necessarily lies in $\I'$ whenever all
its  extensions to a set of rank $n-1$ lie in $\I'$.
%$$
%K_{\min}:=\{I: I\in K, \,\, (K\cap 2^{I})\setminus I=\emptyset \}.
%$$
 %amos
We will use, for $I\in \I(X) \setminus \B(X)$, the notation
$$
\CM(I)\, := \, \{I'\in \I(X) : I\subset\spam(I')\in \CF(X)\}.
$$
Also,
$$
\CI_\varepsilon(X,\I'):=\Ideal\{p_{\eta_F}^{m(F)+\varepsilon(F)} :  F\in \CF(X)\},
$$
where $\varepsilon(F) := 1 \mbox{ if } F = \spam(I) \mbox{ for
some } I\in \I'$, or else $\varepsilon(F) := 0$.

For the proof of our  Theorem \ref{th:28} below,
we will need  the following proposition:

\begin{proposition}[{\cite[Proposition~4.8]{HR}}]\label{pr:21} Let $\CI$ be a
polynomial ideal and let $V$ be a subspace of $\Rn$ of dimension
$d\geq 2$. Let $V_1,\ldots,V_k$ be distinct subspaces of $V$, each
of dimension $d-1$. Suppose that, for $n_1,\ldots,n_k\in \N$, the
ideal $\CI$ contains all homogeneous polynomials defined on $V_i$
of degree $n_i$:
$$
\Pi^0_{n_i}(V_i)\subset \CI.
$$
Then
$$
\Pi^0_N(V)\subset \CI\,\,\,\, \mbox{ whenever }\,\,\,\,
(N+1)(k-1)\geq \sum_{i=1}^k n_i.
$$
\end{proposition}

\begin{theorem}\label{th:28}
Suppose that the semi-external set $\I'$ satisfies the following
additional condition:
$$\hbox{for any $I\in \I(X)\setminus \B(X)$, $\CM(I)\subset \I'$
%amos
implies $I\in \I'$.}$$   Then
\begin{description}
\item[(1)] $\CI_+(X,\I')=\CI_\varepsilon(X,\I')$.
\item[(2)] $ \CP_+(X,\I')=\sum_{I \text{\scriptsize is minimal in }\I'}\left(\bigcap_{Z\in
\BC_I(X)}\CP(X\sqcup Z)\right) $. Here, $I$ is minimal in $\I'$
provided $I\in \I'$ and $(\I'\cap 2^{I})\setminus I=\emptyset$,
and $\BC_I(X)$ denotes all completions of $I$ to a basis:  $
\BC_I(X)\,:=\,\{Z\subset X: I\cup Z\in \B(X)\} $.
\end{description}
\end{theorem}

\begin{proof}
(1)\/ Every generator of $\CI_\varepsilon(X,\I')$ is in  $\CI_+(X,\I')$ directly
from the definition of these ideals, hence
$$
\CI_\varepsilon(X,\I')\subset \CI_+(X,\I').
$$
Therefore, we only need to prove that
$$
\Pi^0_{\#(X\setminus S)}(S\bot)\,\,\subset\,\, \CI_\varepsilon(X,\I'),
\,\,\, \hbox{\rm for all }  S\in S(X,\I') .
$$

\noindent
We run the proof by induction on $n-\dim S$. When $n-\dim S=1$,
i.e., $S\in \CF(X)$, we have
\begin{eqnarray}
 \Pi^0_{\#(X\setminus S)}(S\bot)
\,\,&\subset &\,\, \CI_{\varepsilon}(X,\I'), \quad S\in
S(X,\I'),\label{eq:pr2}\\
\Pi^0_{\#(X\setminus S)+1}(S\bot) \,\,&\subset &\,\,
\CI_{\varepsilon}(X,\I'), \quad S\in \{\spam(I):I\in \I'\}\label{eq:pr1}.
\end{eqnarray}
We will extend now (\ref{eq:pr2}) and (\ref{eq:pr1}) to sets $S$
of lower dimension. For the inductive step, we suppose that
(\ref{eq:pr2}) and (\ref{eq:pr1})
 hold when $n-\dim S=d>0$.
 We now consider the case where $S=\spam I$ for some independent
 $I\in \I(X)$, and $n-\dim S=d+1$.
Consider all possible linear spaces obtained as the span
$$\spam \{ S \cup x \}, \qquad x\in X \setminus S.$$ Assume that $j$ of them
are distinct and label them $S_1$ through $S_j$ (note that $j>1$ since
$d+1>1$).
By our induction hypothesis,
 $\Pi^0_{m_i} (S_i {\perp}) \subset \CI_{\varepsilon}(X,\I')$
for each $i=1, \ldots, j$, where
$m_i := \# (X\setminus S_i) +\varepsilon_i$ and
$\varepsilon_i := 1$ if $S_i \in \I'$ and
$\varepsilon_i := 0$ otherwise, i.e., if $S_i \notin \I'$.
By Proposition~\ref{pr:21}, we conclude that
$$
\Pi_{N}^0(S\bot)\, \subset  \, \CI_{\varepsilon}(X)
$$
whenever $(N+1)(j-1)$ is at least
$$ \sum_{i=1}^j \left(  \# (X\setminus S_i) +  \varepsilon_i \right)=
j \# (X\setminus S) - \# (X\setminus S) +\sum_{i=1}^j \varepsilon_i
\leq (j-1) \# (X\setminus S) + j.$$ If $S$ is spanned by a set that is
in $\I'$, then we can take $N$ to be $\#(X\setminus S)+1$ since
 $N$ so chosen satisfies
$$N \geq \# (X\setminus S) +{1\over j-1}.$$
If $S\in S(X,\I')$, then at least one of the extensions $S_i$ is not in
$\I'$ (otherwise, it will be easy to see that we violate the extra
condition that is now assumed on $\I'$),
and therefore $\sum_{i=1}^j \varepsilon_i \leq j-1$, hence the
value $N=\#(X\setminus S)$ is already large enough for our purposes.
This completes the  inductive step, hence completes the proof of this
part.

(2)\/ We first prove that
$$
\CP_+(X,\I')\subset \sum_{I\, \hbox{\scriptsize is minimal in } \I'}
\left(\bigcap_{Z\in \BC_I(X)}\CP(X\sqcup Z)\right).
$$
Pick $I'\in \I'$. There exists a minimal set $I$ from $\I'$ such that
$I\subset I'$. The definition of $p_{X(I')}$ shows that
$p_{X(I')}\in \CP(X\sqcup Z)$ for all $Z\in \BC_{I}(X)$. We conclude that
$$
p_{X(I')}\in \bigcap_{Z\in \BC_{I}(X)}\CP(X\sqcup Z),
$$
which implies that
$$
\CP_+(X,\I')\subset \sum_{I\, \hbox{\scriptsize is minimal in } \I'}
\left(\bigcap_{Z\in \BC_I(X)}\CP(X\sqcup Z)\right).
$$
We complete the proof by showing that every polynomial $p$ in $\bigcap_{Z\in
\BC_I(X)}\CP(X\sqcup Z)$ lies in $\ker \CI_{\varepsilon}(X,\I')$. Let $p$
be such a polynomial, and let $F\in S(X,\I')\cap \CF(X)$ (we need to
check only this case since the other case, $F \notin S(X,\I')$, is
straightforward).
We need to show that $D_{\eta_F}^{m(F)}p=0$,   where $\eta_F\bot F$ and
$m(F):=\#(X\setminus F)$. Incidentally, note that $I=\emptyset$
is possible only if $\I'=\I(X)$, i.e., when $\CP_+(X,\I')=\CP_+(X)$,
the case when all multiplicities are increased by $1$, which causes
no problem. Thus, $I$ contains at least one element whenever at least
one ``problematic'' multiplicity  exists. We now select
$Z\subset F\cap X$ so that $Z\in \BC_I(X)$ (such a $Z$ exists, since
$\spam I\not\subset F$ and $\codim F=1$). Then, $p\in \CP(X\sqcup
Z)$, hence $p$ is annihilated by $D_{\eta_F}^{\# (X\sqcup Z)\setminus F}$.
Since all vectors of $Z$ lie inside $F$, we conclude that $\# (X\sqcup
Z)\setminus F=m(F)$ and the result follows.  \end{proof}

%%%%%%%%%%%%%%%%%%%% SEMI-INTERNAL SPACES %%%%%%%%%%%%%%%%%%%%%%%%%%%%
\section{Semi-internal zonotopal spaces}
%%%%%%%%%%%%%%%%%%%%%% INTERNAL REVIEW %%%%%%%%%%%%%%%%%%%%%%%%%%%%%%%
\subsection{A review of internal zonotopal spaces}
%%%%%%%%%%%%%%%%%%%%%%%%%%%%%%%%%%%%%%%%%%%%%%%%%%%%%%%%%%%%%%%%%%%%%%

In this section we recall pertinent results from \cite{HR} concerning
the (full) internal zonotopal spaces and their associated ideals.
We impose an (arbitrary but fixed) ordering $\prec$ on $X$. Let
$B\in \B(X)$, and $b\in B$. We say that $b$ {\bf is internally active}
in $B$,  if
$$
b = \max\{X\setminus F\},\,\,\,\, F:=\spam\{B\setminus b\}\in
\CF(X).
$$
A basis $b$ that contains no internally active vectors is
called an {\bf internal basis}. We denote the set of all
internal bases by $\B_-(X)$. We now recall the
definition of the ideal $\CI_-(X)$ from the introduction, and add
the following definitions:
\begin{definition}
\begin{eqnarray*}
 \CJ_-(X)\,&:=&\,\Ideal\{p_Y:
Y\subset X,\ Y\cap B\not=\emptyset,\
\forall B\in \B_-(X)\},\\
\CP_-(X)\,&:=&\,\cap_{x\in X}\CP(X\setminus x),\\
\CD_-(X)\, &:=&\, \ker \CJ_-(X).
\end{eqnarray*}
\end{definition}
We denote
$$V_-:=\{\Cb(B): B\in \B_-(X)\}.$$
The following theorem is taken from \cite{HR}:
\begin{theorem}
\hfill
\begin{enumerate}
\item $\dim \CP_-(X)=\dim\CD_-(X)=\#\B_-(X)$.
\item The map $p\mapsto\left<p,\cdot\right>$ is an isomorphism  between
$\CP_-(X)$ and $\CD_-(X)^\ast$.
\item $\CD_-(X)=\Pi(V_-)$.
\item $\CP_-(X)=\ker\CI_-(X)$.
\item $\CP_-(X)\bigoplus\CJ_-(X)=\Pi$.
\end{enumerate}
\end{theorem}

In contrast with the external and semi-external cases, the
polynomials $\{Q_B:B\in \B_-(X)\}$ do not in general form a basis
for $\CP_-(X)$ \cite{HR}. However, the Hilbert series of
$\CP_-(X)$ can be still determined in the usual way using the
valuation function for internal bases, namely, via the sequence
(see \cite{HR})
$$
h_-(j):=h_{-,X}(j):=\dim(\CP_-(X)\cap \Pi_j^0)=\#\{B\in \B_-(X):
\val(B)=j\}.
$$

%%%%%%%%%%%%%%%%%%%% SEMI-INTERNAL ANALYSIS %%%%%%%%%%%%%%%%%%%%%%%%%%
\subsection{Introduction and analysis of semi-internal zonotopal spaces}
%%%%%%%%%%%%%%%%%%%%%%%%%%%%%%%%%%%%%%%%%%%%%%%%%%%%%%%%%%%%%%%%%%%%%%

Given the set $X$, the internal  space $\cpm{X}$ can be also defined as follows \cite{HR}:
$$\cpm{X}:=\cap_{b\in B}\cp{X\bks b},$$
where $B\in \B(X)$ is arbitrary. In particular, this implies that
the dimension of the above intersection is still expressed in
terms of the matroidal statistics of $X$, i.e., the number of
internally inactive bases of $X$. Suppose, instead, that we choose
$I\in \I(X)$ and define, similarly
$$\cpm{X,I}:=\cap_{b\in I}\cp{X\bks b}.$$
Then, a few questions arise naturally:
\begin{itemize}
\item Does the space $\cpm{X,I}$ depend only on $\spam I$ and not
on $I$ itself?

\item Is there a simple formula that expresses $\dim\cpm{X,I}$ in terms
of the cardinality of a suitable subset of $\B(X)$?

\item Is the ideal $\CI_-(X,I)$ of differential operators that annihilate
$\cpm{X,I}$ generated by powers of the normals to the facets of $X$?

\item Is there a dual construction (on the hyperplane arrangement) of
an ideal of the $\CJ$-class?
\end{itemize}
The answer to all the questions above turns out to be affirmative. In order
to present our results for the above setup, we select a full order
$\prec$ on $X$, and put the vectors in $I$ to be the last ones in this order,
i.e.,
$$x\prec y \quad \hbox{\rm for all } y\in I,\quad x\in X\bks I.$$
Further, we select the following subset of  {\bf $I$-facets:}
$$\CF(X,I):=\{F\in \CF(X): I\not\subset F\}.$$
We then single out the following
subset of $\B(X)$ of {\bf $I$-internal bases}:

\begin{definition}
Let $B\in \B(X)$. We say that $b\in B$ is $I$-internally active (in
$B$) if $b=\max \{X \bks \spam (B\bks b)\}$ and $b\in I$. We say
that $B$ is {\bf $I$-internal} if no vector $b$ in $B$ is
$I$-internally active in $B$.  We denote the set of all $I$-internal
bases by
$$\B_-(X,I).$$
\end{definition}

 Note that if $I\in \B(X)$ then $\cbm{X,I}=\cbm
{X}$, while if $I=\emptyset$ then $\cbm{X,I}=\B{(X)}$. Further,
note that, while $\cpm{X,I}$ does not depend on the order $\prec$,
the set of $I$-internal bases does depend on that order. Note also
that every internal basis is $I$-internal:
$$\B_-(X)\subset \B_-(X,I).$$
In addition to the $I$-internal bases, we need the following ideal:
$$\cim{X,I}:=\Ideal\{\CI(X)\cup\{\eta_F^{m(F)-1}: \
F\in \CF(X,I),\ 0\not=\eta_F\perp F\}\}.$$

\begin{theorem}\label{th:minus1}

\hfill
\begin{enumerate}
\item $\dim\cpm{X,I}=\#\B_-(X,I)$.
\item $\cpm{X,I}=\ker \cim{X,I}$.
\end{enumerate}
\end{theorem}

\noindent The result $\cpm{X,I}=\ker \cim{X,I}$ implies  that
$\cpm{X,I}$ depends only on $\spam I$ and not on $I$ itself. We
prove this theorem in the sequel. Let us next introduce the dual
setup, that goes as follows: first, we say that $Y\subset X$ is
{\bf $I$-long} if $Y\cap B\not=\emptyset$ for each $B\in
B_-(X,I)$. Set
$$\cjm{X,I}:=\Ideal\{p_Y:\ Y\subset X \hbox{ is $I$-long}\}.$$
Note that the ideal $\cjm{X,I}$ depends on the order we choose,
since the set $B_-(X,I)$ depends on that order.

Let now $\CH(X,\lam)$ be a simple hyperplane arrangement as in
Section~\ref{sec:motiv}. Then there is a natural bijection $B\mapsto v_B$ from $\B(X)$ onto
the vertex set $V(X,\lam)$ of the hyperplane arrangement. Denote
$$V_-(X,\lam,I):=\{v_B: \ B\in\B_-(X,I)\},\quad \cdm{X,I}:=\ker\cjm{X,I}.$$

\begin{theorem}\label{th:minus2}

\hfill
\begin{enumerate}
\item $\cdm{X,I}=\Pi(V_-(X,\lam,I))$, in particular
$$\dim\cdm{X,I}=\#\B_-(X,I).$$
\item $\cjm{X,I}\oplus \cpm{X,I}=\Pi$.
\item The map $p\mapsto \inpro{p,\cdot}$ is an isomorphism from $\cpm{X,I}$
onto $\cdm{X,I}^\ast$.
\end{enumerate}
\end{theorem}

\begin{proof} (Theorems \ref{th:minus1} and \ref{th:minus2}). \/ The last
assertion in Theorem \ref{th:minus2} is a direct consequence of the second
assertion in that theorem and the fact that $\cdm{X,I}=\ker \cjm{X,I}$:
the argument is identical to the one used to prove {\bf 5.} of Theorem
\ref{thm 2.6}. We
divide the rest of the proof into six parts as follows.

{\bf Part I: $\cpm{X,I}=\ker\cim{X,I}$.}\hfill\break
Since, for every $x\in X$,
$\CP(X\setminus x)=\ker \CI(X\setminus x)$, we may prove the
stated result by showing that (i)
$\CI(X\setminus x)\subset \cim{X,I}$ for every $x\in I$, and (ii)
$\cim{X,I}\subset \Ideal\{\cup_{x\in I} \CI(X\setminus x) \}$.

For the proof of (i), fix $x\in I$, and denote $X':=X\setminus x$. A
generator $Q$ in the ideal $\CI(X')$ is of the form
$Q:=p_{\eta_F}^{m_{X'}(F)}$, with $F\in \CF(X')$. Then $F$
is also a facet hyperplane of $X$. Now, if $I\subset F$, then
$x\in F$. Therefore
$m_{X'}(F)=m_X(F)$ and $Q$ above lies in $\CI(X)$ hence also in $\cim{X,I}$.
If, on the other hand, $I\not\subset F$, then the polynomial
$p_{\eta_F}^{m_X(F)-1}$ lies in $\cim{X,I}$. This implies that
$Q$ lies in that ideal, too, since $m_{X'}(F))\ge m_X(F)-1$.

For (ii), we first note that $\CI(X)$ lies in each ideal of the form
$\CI(X\bks x)$. Thus, we  may simply show that every generator of
$\cim{X,I}$ of the form $Q=p_{\eta_F}^{m(F)-1}$ lies in one of the ideals
$\CI(X\bks x)$, $x\in I$. Here, $F$ is a facet hyperplane of $X$, and
$I\not\subset F$. Let $x\in I\bks F$.
Denote $X':=X\setminus x$.
Since $x\notin F$, it is clear that $F\in \CF(X')$, and then
$m_{X'}(F)=m_X(F)-1$. Thus the polynomial $Q$ lies
in $\CI(X')$, and (ii) follows.

\smallskip
{\bf Part II: $\Pi(V_-(X,\lam,I))\subset \cdm{X,I}$,
$\dim\cdm{X,I}\ge \#\B_-(X,I)$.} Both claims are obtained by a standard
argument; see Theorem~\ref{th:pi}.

\smallskip
{\bf Part III: $\dim\cdm{X,I}=\#\B_-(X,I)$.} In view of Part II, we only
need to prove the $\le $ inequality.

To this end, we note that $\CJ(X)+\CP(X)=\Pi$ by Theorem~\ref{th:exzono}
and since
$\CJ(X)\subset\cjm{X,I}$, we have that $\cjm{X,I}+\CP(X)=\Pi$. Now,
let $(Q_B)_{B\in\B(X)}$ be the homogeneous basis for $\CP(X)$ (per
our chosen order for $X$; see Theorem~\ref{th:basis}). Set
$$\CP_{ex}:=\spam\{Q_B: \ B\in \B(X)\bks \B_-(X,I)\}.$$
We show that $\CP_{ex}\subset \cjm{X,I}$. This will imply that
$$\Pi=\cjm{X,I}+\CP_{in},\quad \CP_{in}:=\spam\{Q_B:\ B\in \B_-(X,I)\},$$
hence that
$$\dim\cdm{X,I}=\dim\Pi/\cjm{X,I}\le \dim\CP_{in}=\#\B_-(X,I),$$
which is the desired result.

So, we need to show that each $Q_B$, $B\in\B(X)\bks\cbm{X,I}$ lies
in $\cjm{X,I}$. Now, $Q_B=p_{Y}$, with
$$Y:=\{x\in X\bks B:\ x\not\in\spam\{b\in B: b\prec x\}\}.$$
Since $B\not\in \cbm{X,I}$, there exists $b\in B\cap I$ such that, with
$F:=\spam(B\bks b)$, $b=\max\{X\bks F\}$.  This shows that
$$X\bks Y\subset F\sqcup\{b\}.$$
However, every basis $B\subset F\sqcup\{b\}$ is a basis for $F$ augmented
by $b$, hence is not $I$-internal. Consequently, $Y$ is $I$-long
hence lies in $\cjm{X,I}$.

\smallskip
{\bf Part  IV: $\cdm{X,I}=\Pi(V_-(X,\lam,I))$.} This follows directly from
Parts II and III, since \break $\dim\Pi(V_-(X,\lam,I))=\#\cbm{X,I}$.

\smallskip
{\bf Part  V: $\dim\cpm{X,I}\le\#\cbm{X,I}$.}
The proof of this assertion follows  from the fact that
\begin{equation}\label{int:pmex}
\cpm{X,I}\cap\CP_{ex}=\{0\}.
\end{equation}
Indeed, once (\ref{int:pmex}) is proved, we conclude that,
since $\cpm{X,I},\CP_{ex}\subset \CP(X)$,
$$\dim\cpm{X,I}\le \dim\CP(X)-\dim\CP_{ex}=\dim\CP_{in}=\#\cbm{X,I}.$$

The actual proof of (\ref{int:pmex}) follows almost verbatim the
proof of the special case $I\in \B(X)$ from \cite[Theorem
5.8]{HR}. We briefly outline the proof there, and add an
additional argument that is required in our more general setup.

We start by writing down an arbitrary element in $\CP_{ex}\bks 0$:
$$\sum_{B\in\B(X)\bks \cbm{X,I}}a(B)Q_B.$$
We then select a summand $a(B')Q_{B'}$ in the above sum such that
$a(B')\not=0$, and such that $B'$ is minimal (among all summands with non-zero
coefficients) with respect to the valuation
$$\alp(B):=\#M(B), \quad M(B):=
\{b\in B\cap I: \ b=\max(X\bks \spam(B\bks b))\}.$$
By the definition of $\CP_{ex}$, $\alp(B')>1$. Let $b'\in M(B')$, and
set $F':=\spam(B'\bks b')$. The argument in \cite[Theorem 5.8]{HR} then
reduces the proof of the fact that $a(B')=0$ to
showing that, if $B\in \B(X)\bks \cbm{X,I}$, if $B\cap F'= B'\cap F'$, and
if $B\not=B'$, then $\alp(B)<\alp(B')$. So, we pick now such $B$, and
prove that $M(B)\subset M(B')$. This proves that $\alp(B)<\alp(B')$, since
$b'\in M(B')\bks M(B)$ (if $b'\in M(B)$  it follows that $B=B'$).

Thus, we pick $x\in M(B)$ and prove that it lies in $M(B')$, too.
To this end, we denote $A:=B'\bks b'$. Then $A$ is a basis for $F'$,
and $B=A\sqcup b$, for some $b\in X$. Necessarily, $x\in A$. Set
$S:=A\bks x$.
Note that $\rank S=n-2$. Assume that $x\not\in M(B')$. Since
$x\in I\cap B'$,  we conclude that there
exists $y \succ x$ such that $y\not\in
\spam\{S\cup b'\}$. Assume $y$ to be maximal element outside $\spam\{S\cup
b'\}$. We get the contradiction to the existence of such $y$ by
showing that it is impossible to have $y \succ b'$, and it is also
impossible to have $y \prec b'$. Note that $y\in I$, since $x\in I$.

If $y\succ b'$, then, since $b'$ is maximal outside $\spam\{B\bks b'\}
=\spam A=\spam\{S\cup x\}$, we have that $y\in\spam\{S\cup x\}$. Also,
since $y\succ x$, and $x$ is maximal outside $\spam \{ B\bks x\}=
\spam\{S\cup b\}$, we have $y\in\spam\{S\cup b\}$. But $S\cup b\cup x=B$,
and $B$ is independent, hence $y\in \spam S$, which is impossible since
we assume $y$ to be outside $\spam\{S\cup b'\}$.

Otherwise, $y\prec b'$. The maximality of $y$ then implies
that $x\prec y \prec b'$. The maximality of $x$ outside $\spam\{S\cup b\}$
implies that $b'\in\spam\{S\cup b\}$. Since $b'\not\in S$,
we obtain that $\spam\{S\cup b\}=\spam\{S\cup b'\}$, which is impossible
since $y$ lies in exactly one of these two spaces.

\smallskip
{\bf Part  VI: $\dim\cpm{X,I}=\#\cbm{X,I}$, and $\cjm{X,I}\oplus
\cpm{X,I}=\Pi$.} We prove in Lemma \ref{lem:sum}  below that
$$\cjm{X,I}+\cpm{X,I}=\Pi.$$
This implies that
$$\dim\cpm{X,I}\ge \dim\ker\cjm{X,I}=\#\cbm{X,I},$$
with the equality by Part III. This, together with Part V shows that
$\dim\cpm{X,I}=\#\cbm{X,I}.$ Thus, $\dim\cpm{X,I}=\dim\Pi/\cjm{X,I}$,
hence the sum $\cpm{X,I}+\cjm{X,I}=\Pi$ must be direct.
\end{proof}

\begin{lemma}  \label{lem:sum}
$$\cjm{X,I}+\cpm{X,I}=\Pi.$$
\end{lemma}

\begin{proof}
The special case of this result for the choice $I\in \B(X)$ was proved
in \cite[Theorem 5.7]{HR}. While most of the proof here
parallels the one in \cite{HR}, there is a significant difference in one
of the details which requires us to provide here a complete self-contained
proof.

The proof of the previous theorem reduces the proof here to showing that,
for each $Q_B$, $B\in\cbm{X,I}$, $Q_B\in \cjm{X,I}+\cpm{X,I}$.
Fixing $B\in \cbm{X,I}$, we know that
$Q_B=p_{X(B)}$, for  suitable $X(B)\subset X$. We decompose
$X(B)$ in a certain way $X(B)=Z\sqcup W$. Thus
$$Q_B=p_Zp_W.$$
We then replace each $w\in W$ by a vector $w'$ (not necessarily
from $X$), to obtain a new polynomial
$$\til{Q_B}\eqbd p_{Z}p_{{W'}},$$
and prove that (i) $\til{Q_B}\in \cpm{X,I}$, and (ii)
$Q_B-\til{Q_B}\in \cjm{X,I}$.

So, let $Q_B=p_{X(B)}$ be given. If $Q_B\in\ker \cim{X,I}=\cpm{X,I}$,
there is nothing to prove. Otherwise, let ${\bf F}\subset\CF(X)$ be
the collection of {\it all\/} facet hyperplanes $F$ for which
$D_{\eta_F}^{m(F)-1}Q_B\not=0$, and
$\max(X\bks F)\in I$. The set ${\bf F}$ is not empty, since otherwise
$Q_B\in \ker \cim{X,I}$. Given $F\in {\bf F}$, we conclude
that $\#(X(B)\bks F)\ge m(F)-1$, hence that, with $Y\eqbd X\bks X(B)$,
$\#(Y\bks F)\le 1$. Since $B\subset Y$, the set
$Y\bks F$ must be a singleton $x_F\in B$. We denote
$$X_\bfF\eqbd \{x_F:\ F\in \bfF\}.$$

Define
$$W\eqbd \{\max \{ X\bks F \} :\ F\in\bfF\}.$$
Then $W\subset I$, by the definition of $\bfF$.
We index the vectors in $W$ according to their order in $X$:
$W=\{w_1\prec w_2\prec\ldots\prec w_k\}.$
For each $1\le i\le k$, we define
$$X_i\eqbd \{x_F:\ F\in\bfF,\ \max\{X\bks F\}= w_i\},\quad \bfF_i\eqbd \{F\in
\bfF:\ x_F\in X_i\}.$$
Thus, $X_\bfF=\bigcup_{i=1}^kX_i$.

Setting all these notations, we first observe that $W\cap X_\bfF=\emptyset$,
i.e., $w_i$ does not lie in $X_i$. Indeed, the set $X_\bfF$ is a subset
of every $B'\in \B(Y)$, with $\spam(B'\bks x_F)=F$ for each $x_F\in X_\bfF$.
If some $x_F$ is $\max\{ X\bks F\}$, it will be $I$-internally active in every
$B'\in\B(Y)$, which would imply that $\B(Y)$ does not contain $I$-internal
bases, which is impossible since $B\in \B(Y)$. Thus, $W\subset X(B)$, and we
define $Z\eqbd X(B)\bks W$, to obtain
$$Q_B=p_Zp_W.$$

Define further:
$$S_i\eqbd \cap\{F: F\in\cup_{j=1}^i \bfF_j\},\quad S_0\eqbd \Rn.$$
Then, for $i=1,\ldots,k$,
$S_{i-1}=S_{i}\oplus \spam\, X_{i}$,
and $w_{i}\in S_{i-1}\bks S_{i}$.
Thus, for $i=1,\ldots k$, the vector $w_i$ admits a unique representation of the
form
\begin{equation}\label{defxijprime}
w_i=w_i'+\sum_{x\in X_{i}}a_xx,\quad w_i'\in S_{i},\  a_x\in \R\bks
\{0\}.
\end{equation}
Define
$$W'=\{w_1',\ldots,w_k'\},\ \hbox{and }
\tilQ_B\eqbd p_{Z}p_{W'}.$$
We prove first that
$$\tilQ_B-Q_B=p_{Z}(p_{W'}-p_W)$$
lies in $\cjm{X,I}$.
To this end, we multiply out the product
\begin{equation}
p_{W'}=\prod_{i=1}^k p_{w_i'}=\prod_{i=1}^k(p_{w_i}-\sum_{x\in X_i
}a_x p_{x}).\end{equation}
Every summand in the above expansion is of the form $p_\Xi$, with
$\Xi$ a suitable mix of $W$-vectors and $X_\bfF$-vectors.
The summand $p_W$ in the above expansion in canceled when we subtract
$Q_B$. Any other $\Xi$ is obtained from $W$ by replacing at least once
a  $w_i$ vector by some vector  in $X_i$, which we denote by $x_i$. Let
$w_{i_1}\prec w_{i_2}\prec\ldots\prec w_{i_j}$ be all the $w$-vectors
in $W\bks \Xi$, and let $F_1$ be the facet hyperplane that
corresponds to $x_{i_1}$ ($F_1\eqbd \spam(B\bks x_{i_1})$.)
Then, we have that $w_{i_1}\in X\bks(Z\sqcup \Xi)=:Y'$, and we claim that
$Y'\bks w_{i_1}\subset F_1$. To this end, we write
$Y'\bks F_1=((Y'\cap Y)\bks F_1)\sqcup (Y'\bks Y)\bks F_1$.
Now, $Y\bks F_1=x_{i_1}$, and since $x_{i_1}\not\in Y'$ (as it was replaced
by $w_{i_1}$), the term $(Y'\cap Y)\bks F_1$ is empty. The second
term consists of $(w_{i_m})_{m=1}^j\bks F_1$. However,
$w_{i_m}\in S_{i_m-1}\subset S_{i_1}\subset F_1$, for every $m\ge 2$.
Thus, $w_{i_1}$ is the only vector in $Y'\bks F_1$. Being also the
last vector in $X\bks F_1$, we conclude that $w_{i_1}$ is
$I$-internally active in every $B\in \B(Y')$, hence
that
$p_{Z\sqcup \Xi}\in \cjm{X,I}$. This being true for every summand
in $\tilQ_B-Q_B$, we conclude that this latter polynomial lies in
$\cjm{X,I}$.

We  now prove that $\tilQ_B=p_{Z\sqcup W'}\in \ker\cim{X,I}$. To this end,
we need to show that, for every $F\in\CF(X)$,
$\#((Z\sqcup W')\bks F)<m(F)-\eps(F)$, with
$\eps(F)=1$ if $I\not\subset F$, and $\eps(F)=0$ otherwise.
We divide the discussion here to three cases.
As before, $Y\eqbd X\bks X(B)$.

Assume first that $F\in \bfF_i$ for some
$1\le i\le k$. In this case, $\eps(F)=1$.
Now, for $X(B)=Z\sqcup  W$ we had that
$\#((Z\sqcup W)\bks F)=m(F)-1$. Also, $x_F$ is the only
vector in $Y\bks F$, and $x_F\in X_i$.  Thus,
the subset $X_j\subset Y$, must lie in $F$ for every $j\not= i$,
which means that we conclude that,
$w_j\in F$ iff $w'_j\in F$ (since
$w_j-w_j'\in \spam X_j\subset F$).
Finally, while $w_i\not\in F$, $w'_i\in S_i\subset F$,
hence, altogether, $\#(W'\bks F)<\#(W\bks F)$,  and we
reach the final conclusion that
$$\#((Z\sqcup W')\bks F)<
\#((Z\sqcup W)\bks F)=m(F)-1.$$

Secondly, we assume $F\in\CF(X)\bks \bfF$, but still
that $S_k\subset F$.
Let $j\ge 1$ be the minimal index $i$ for which $S_i\subset F$. Define:
$$m_1\eqbd \#\{w\in W\bks w_j: w\in F \;\; {\rm and} \;\, w'\not \in F\}, \hbox{ and }
m_2\eqbd \#(\sqcup_{i\not=j}(X_i\bks F)).  $$
Note that (since $w_j'\in F$)
$\#((Z\sqcup W')\bks F)\le m(F)+m_1-m_2-\#((X_j\sqcup w_j)\bks F)$.
Note further that for $i\not=j$, if
$w_i'\not\in F$, while $w_i\in F$, then, since
$w_i-w_i'\in \spam X_i$, we have that $\#(X_i\bks F)>0$.
Thus, $m_1\le m_2$.
In addition, for $i=j$,
$$w_j-w'_j\in\spam X_j.$$
We know {\it a priori\/} that $S_j\oplus \spam X_j=S_{j-1}$.
Since $S_j\subset F$, while $S_{j-1}\not\subset F$,
we must have that $X_j\bks F\not=\emptyset$. But, $w_j'\in F$,
hence $\#((X_j\sqcup w_j)\bks F)\ge 2$. Thus,
$$\#((Z\sqcup W')\bks F)\le m(F)+m_1-m_2-\#((X_j\sqcup w_j)\bks F)<m(F)-1.$$

Lastly, assume that $S'\eqbd S_k\cap F\not= S_k$.  We define now,
similarly,
$$m_1\eqbd \#\{w\in W: w\in F\wedge w'\not \in F\}, \hbox{ and }
m_2\eqbd \#(\sqcup_{i=1}^k(X_i\bks F)).  $$
Then, as before, $m_1\le m_2$. Hence,
$\#((Z\sqcup W')\bks F)\le m(F)-\#U$, with
$U\eqbd (Y\cap S_k)\bks S'$. Note that
all the vectors of $Y\bks U$ lie in the rank deficient
set $(Y\cap S')\sqcup (Y\bks S_k)$), hence $U$ is not empty.
If $\#U\ge 2$, or if $m_1<m_2$, we are done since it follows that
$\#((Z\sqcup W')\bks F)\le m(F)-2$. However, if $U$ is a singleton and
$m_1=m_2$,
our analysis only shows that $\#((Z\sqcup W')\bks F)\le m(F)-1$. That means
that, for this case, we either need to furnish a finer estimate, or show that
$I\subset F$.  We prove the latter.  The argument below uses the
following approach: after realizing that the singleton $U$ lies in the
basis $B$, we define $F':=\spam(B\bks U)$, and conclude that
$F'$ must contain $I$. We then invoke the condition
$m_1=m_2$ in order to obtain a spanning set for $F$ by removing
from $F'\cap Y$ all the vectors in $F'\cap X_{\bfF}$ and adding instead vectors
from $W\subset I$. In this way, we guarantee that $I\subset F$, too.

Here are the details.
Let $F'$ be the hyperplane spanned by  $X\bks (X(B)\sqcup U)$.
This hyperplane is spanned by elements of $B$. Moreover,
$(X\bks X(B))\bks F'=U$, and $U$ is a singleton. At the same time,
$F'$ is not listed in $\bfF$ (since $U\in S_k$ and $S_k$ is disjoint
of $X_{\bfF}$).
Then, necessarily, $I\subset F'$, hence also $W\subset F'$.

Next, let $J\subset \{1,\ldots,k\}$ be defined by
$$j\in J \iff (w_j\in F\hbox{ and } w'_j\not\in F).$$
The equality $m_1=m_2$ implies that $X_j\bks F$ is a singleton
$x_j$ for $j\in J$ and is empty otherwise. Now, we know that
$Y':=Y\bks U$ spans $F'$. $Y'':=Y'\bks \{x_j:\ j\in J\}$ is a
subset of $F\cap F'$ of rank $n-\#J-1$. Since $W\subset F$, we
have that $Y''':=Y''\sqcup \{w_j: j\in J\}\subset F$. Since each
$w_j$, $j\in J$, is independent of $(Y\bks x_j)\sqcup\{w_i: i>j\}$,
we conclude that $\rank Y'=\rank Y'''$, hence\footnote{If a set
$A$ in a vector space $V$ contains a subspace $V'$ in its span,
and if $A'=v'\sqcup (A\bks a)$ for some $a\in A$ and $v'\in V'$,
then either $V'\subset \spam A$, or else $\rank A'<\rank A$.}
 $I\subset \spam Y'''$. Thus, $I\subset F$, and our proof is complete.
\end{proof}

We now consider the Hilbert series of $\cpm{X,I}$, i.e.,
$$h_{-,I}(j):=\dim (\cpm{X,I}\cap \Pi_j^0).$$
 In general, it is not
true that the polynomials $Q_B:=p_{X(B)}, B\in \B_-(X,I),$ form a
basis for $\cpm{X,I}$. However, they {\em can } be used for
computing $h_{-,I}(j)$. In fact, we have
$$
h_{-,I}(j)=\#\{B\in \B_-(X,I):\val(B)=\deg Q_B=j\}.
$$
We observe this fact from the proof of Lemma \ref{lem:sum}:
Every $Q_B$ there was proved to be writable as
$$
Q_B=\tilde{Q}_B+f_B
$$
with  $\tilde{Q}_B\in \cpm{X,I}$ and $f_B\in \cjm{X,I}$.
The fact $\tilde{Q}_B, B\in \B_-(X,I),$ are independent follows directly
from the independence of $Q_B, B\in \B_-(X,I)$, and the fact that
the sum $\cjm{X,I}+\spam\{Q_B: B\in \B_-(X,I)\}$ is direct (from
Part III of the proof of Theorems \ref{th:minus1} and
\ref{th:minus2}, and the fact  $\dim\ker \cjm{X,I}=\#\B_-(X,I)$),
which implies that $\{\tilde{Q}_B: B\in \B_-(X,I)\}$ is a basis
for $\cpm{X,I}$. Note that each $\tilde{Q}_B$ is obtained by
replacing some of the factors $p_\omega, \omega\in X$ of $Q_B$, by
polynomials $p_{\omega'}, \omega'\in \R^n\setminus 0$. Thus, $\deg
\tilde{Q}_B=\deg Q_B=\val(B)$, hence we may indeed compute
$h_{-,I}(j)$ via the polynomials $Q_B, B\in \B_-(X,I)$.

\begin{remark}
We note that,  if $\# I\leq 2 $, then
\begin{equation}\label{eq:cpmbuil}
\cpm{X,I}=\cpm{X}+\spam\{Q_B: B\in \B_-(X,I)\setminus \B_-(X)\}.
\end{equation}
In general, however, (\ref{eq:cpmbuil}) is not
valid for $ \# I\geq 3$.
\end{remark}

%%%%%%%%%%%%%%%%%%%%%%%%%%%%%  BIBLIOGRAPHY %%%%%%%%%%%%%%%%%%%%%%%%%%%%%%%%%%%%

\bibliographystyle{amsplain}

\end{document}